\documentclass[a4paper]{amsart}
\usepackage{amssymb}
\author[I.~V.~Arzhantsev]{Ivan V.~Arzhantsev}
\address{
Department of Higher Algebra\\ Faculty of Mechanics and
Mathematics\\ Moscow State University\\ Leninskie Gory, GSP-1, Moscow, 119991, Russia}
\email{arjantse@mccme.ru}
\urladdr{http://mech.math.msu.su/department/algebra/staff/arzhan.htm}
\author[A.~P.~Petravchuk]{Anatoliy P. Petravchuk}
\address{
Algebra Department \\ Faculty of Mechanics and Mathematics\\ Kyiv
Taras Shevchenko University\\ 64, Volodymyrskaia street, 01033 Kyiv,
Ukraine} \email{aptr@univ.kiev.ua}
\title[Saturated subfields and invariants of finite groups]
{Saturated subfields and invariants \\ of finite groups}
\thanks{Supported by grants NSh-1983.2008.1, DFFD F25.1/095 and RFFI 09-01-90416}
\newcommand{\kk}{\Bbbk}

\newcommand{\RR}{\mathbb{R}}

\newcommand{\CC}{\mathbb{C}}
\newcommand{\PP}{\mathbb{P}}

\newcommand{\GL}{\mathop{\mathrm{GL}}}
\newcommand{\SL}{\mathop{\mathrm{SL}}}

\newtheorem{theorem}{Theorem}
\newtheorem{corollary}{Corollary}

\newtheorem{lemma}{Lemma}
\theoremstyle{definition}

\newtheorem{example}{Example}

\theoremstyle{remark}

\theoremstyle{problem}

\begin{document}

\sloppy

\begin{abstract}
Every subfield $\kk(\phi)$   of the field of rational functions
$\kk(x_1,\dots,x_n)$ is contained in a unique maximal subfield of
the form $\kk(\psi)$. The element $\psi$ is called generative for
the element $\phi$. A subfield of $\kk(x_1,\dots,x_n)$ is called
saturated if it contains a generative element of each its
element. We study the saturation property for subfields of invariants $\kk(x_1,\dots,x_n)^{G}$,
where $G$ is a finite group of automorphisms of the field $\kk(x_1,\dots,x_n)$.
\end{abstract}

\maketitle

\section{Introduction}
Consider the polynomial ring $\kk[x_1,\dots,x_n]$ over a
field $\kk$. Recall that a polynomial $h\in
\kk[x_1,\dots,x_n]\setminus \kk$ is called {\it closed} if the
subalgebra $\kk[h]$ is integrally  closed  in
$\kk[x_1,\dots,x_n]$. For any polynomial $f \in
\kk[x_1,\dots,x_n]\setminus \kk$ there exist a closed polynomial
$h$ and a polynomial $F(t)\in\kk[t]$ such that $f=F(h)$. The
polynomial $h$ is determined by the polynomial $f$ up to
affine transformations $h\to \alpha h+\beta, \  \alpha \in
\kk^{\times}, \beta \in \kk,$ and is called {\it generative} for
the polynomial $f$.

A subalgebra $A\subseteq \kk[x_1,\dots,x_n]$ is said to be 
{\it saturated} if it contains generative polynomials of every its
nonconstant element. Clearly, any subalgebra $A$ which
is integrally closed in $\kk[x_1,\dots,x_n]$ is also saturated, but
the converse is not true. In the paper \cite{AP}, we studied,
among others, the saturation property for subalgebras of
invariants $\kk[x_1,\dots,x_n]^{G}$ of finite subgroups
$G\subseteq GL_{n}(\kk).$ 

\begin{theorem}\cite[Thm.~2]{AP} \label{T1} 
Let $\kk$ be a
field and $G\subseteq GL_{n}(\kk)$ be a finite subgroup. Then the
subalgebra of invariants $\kk[x_1,\dots,x_n]^{G}$ is saturated in
$\kk[x_1,\dots,x_n]$ if and only if the subgroup $G$ admits no
non-trivial homomorphisms $G\to \kk^{\times}.$
\end{theorem}

In particular, the saturation property for the subalgebra
$\kk[x_1,\dots,x_n]^{G}$ depends on the field $\kk$ and on the
(abstract) group $G$, but does not depend on the matrix
realization of $G$.

The aim of this note is to investigate the saturation property 
for subfields of rational invariants. We start with some
definitions. Let $\kk(x_1,\dots,x_n)$ be the field of
rational functions and $\psi \in \kk(x_1,\dots,x_n)\setminus \kk$.
Recall that a subfield $L\subseteq \kk(x_1,\dots,x_n)$ is said
to be  {\it algebraically closed} in $\kk(x_1,\dots,x_n)$ if every
element of $\kk(x_1,\dots,x_n)$ which is algebraic over $L$
belongs to $L$. A rational function $\psi$ is called {\it closed}
if the subfield $\kk(\psi)$ is algebraically closed in
$\kk(x_1,\dots,x_n)$.  For any $\phi \in
\kk(x_1,\dots,x_n)\setminus \kk$ there exist a closed function
$\psi$ and an element $H(t)\in \kk(t)$ such that $H(\psi)=\phi$.
The element $\psi$ is determined up to transformations
$$ \psi \to \frac{\alpha \psi +\beta}{\gamma \psi +\delta}, \ \ \
\alpha , \beta , \gamma ,\delta \in \kk, \ \ \  \alpha \delta -\beta
\gamma \not= 0, \hfill \eqno (1)$$
see \cite[Thm.~6]{Ol} or \cite[Lemma~2]{PI}, and is called
{\it generative} for the element $\phi$.

Generative elements appear naturally in different problems. Let $\phi=\frac{P}{Q}$
be an uncancelled form of a rational function $\phi$. Consider 1-form $\omega=QdP-PdQ$. Its {\it field of constants} is a field of
rational functions $\xi\in\kk(x_1,\dots,x_n)$ satisfying $\omega\wedge d\xi=0$. It is known that this field is generated by the generative
element $\psi$: for two variables it goes back to classical works of A.~Poincar\'e on algebraic integration of
differential equations, and the general case is discussed in~\cite{Ol}. Further, if the kernel
of a $\kk$-derivation of the field $\kk(x_1,\dots,x_n)$ is one-dimensional, then it has a form 
$\kk(\psi)$ for some closed rational function $\psi$. Conversely, in~\cite{EON} for any closed rational function $\psi$ 
a $\kk$-derivation of the field $\kk(x_1,\dots,x_n)$ with the kernel $\kk(\psi)$ is constructed. 
Some characterizations of closed rational functions are obtained in \cite{Bo} and \cite{PI}. In \cite{Bo}, 
an analogue of Stein's theorem for rational functions is proved and an estimate of the number of reducible
fibers is obtained. Finally, in \cite[Sec.~7]{Ol} an algorithm that finds a generative element $\psi$ for a given
rational function $\phi$ is given. 

A subfield $L\subseteq \kk(x_1,\dots,x_n)$ is said to be {\it
saturated} if for any $\phi \in L\setminus \kk$ the generative
function $\psi$ of $\phi$ is contained in $L$. Clearly, 
every algebraically closed subfield of
$\kk(x_1,\dots,x_n)$ is saturated. The results below imply that the
converse is not true.

Let $G$ be a finite group of automorphisms of the
field $\kk(x_1,\dots,x_n)$. We study the saturation property for 
the subfield of invariants $\kk(x_1,\dots,x_n)^{G}$. It is
interesting to remark that, as follows from L\"uroth's Theorem, 
the subfield $\kk(x_1,\dots,x_n)^{G}$ is saturated in
$\kk(x_1,\dots,x_n)$ if and only if every one-dimensional
$G$-invariant subfield of $\kk(x_1,\dots,x_n)$ lies in
$\kk(x_1,\dots,x_n)^{G}$. Geometrically this means that 
every $G$-equivariant rational morphism from $\kk^{n}$ to a
curve is $G$-invariant.


\section{Main results}\label{s3}

The next result yields a sufficient condition for the subfield of
invariants to be saturated.

\begin{theorem}\label{T2}
Let $\kk$ be a field of characteristic zero and $G$ be a finite
group of automorphisms of the field $\kk(x_1,\dots,x_n)$.
Suppose that $G$ has neither nontrivial abelian factor
groups nor factor groups which are isomorphic to the alternating
group $A_{5}.$  Then the subfield $\kk(x_1,\dots,x_n)^{G}$ is
saturated in $\kk(x_1,\dots,x_n)$.
\end{theorem}

\begin{proof}
Suppose that the generative function $\psi$ of an element $\phi \in
\kk(x_1,\dots,x_n)^G\setminus \kk$ is not invariant. It follows from
(1) that the element $\psi$ is determined by
$\phi$ up to the action of the group
$\rm{PGL_{2}}(\kk)$. Therefore the group $G$ admits a non-trivial
homomorphism to $\rm{PGL_{2}}(\kk)$. This homomorphism is
defined over a finitely generated extension of the field of
rational numbers. Any such extension is isomorphic to a
subfield of the field of complex numbers $\CC$. On the other hand,
from the classification of finite subgroups in $\rm{PSL}_{2}(\CC
)=\rm{PGL}_{2}(\CC)$ (see, for example, \cite[4.4]{Sp})
it follows that every such subgroup is either solvable or
isomorphic to the group $A_{5}$. This contradiction completes the proof
of the theorem.
\end{proof}

\begin{example}\label{ex1}
The automorphism group of the field $\CC(x)$ is isomorphic to
$\rm{PSL}_{2}(\CC).$ Let $G\subset \rm{PSL}_{2}(\CC)$ be the group
of rotations of icosahedron, which is isomorphic to $A_{5}$. Then
the subfield $\CC(x)^{G}$ is not saturated in $\CC(x)$.
\end{example}

Now we consider subfields of invariants which correspond to 
(regular) representations of a finite group $G.$  Note that the
group $\rm{PSL}_{2}(\CC)$ contains a unique (up to conjugation)
subgroup isomorphic to $A_5$. Denote its preimage in
$\rm{SL}_{2}(\CC)$ by $I_{120}.$ This is a group of order $120$.
Since $A_{5}$ admits no non-trivial two-dimensional
representations, the subgroup $I_{120}$ has no  subgroups
isomorphic to $A_{5}$. The center $Z$ of $I_{120}$ consists of
two elements, and $I_{120}/Z\simeq A_{5}.$ Therefore the group
$I_{120}$ coincides with its commutator subgroup.

\begin{theorem}\label{th3}
Let $\kk$ be an algebraically closed field of characteristic zero
and $G\subset \GL_n(\kk)$ be a finite subgroup. The
following conditions are equivalent:

(i) the group $G$ has  neither nontrivial abelian factor groups
nor factor groups isomorphic to $I_{120};$

(ii) the subfield $\kk(x_1,\dots,x_n)^{G}$ is saturated in
$\kk(x_1,\dots,x_n)$.

\end{theorem}

\begin{proof}

(i)$\Rightarrow$(ii)\ Suppose that there exists an element $\phi \in
\kk(x_1,\dots,x_n)^G\setminus \kk$ whose generative element $\psi
=\frac{p}{q}, \ p, q\in \kk[x_1,\dots,x_n]$  is not invariant.
Since the subfield $\kk(\psi)$ is $G$-invariant, so is the
subspace generated by $p$ and $q$. Thus the group $G$ admits a
nontrivial homomorphism to the group $\rm{GL}_{2}(\kk).$ By
assumptions, the image of $G$ is contained in  the subgroup
$\rm{SL}_{2}(\kk)$. Further, the image of $G$ in $\text{Aut}(\kk(\psi))$
should be a subgroup isomorphic to $A_5$.
Therefore the image of $G$ in $\SL_2(\kk)$ coincides with
$I_{120}$, which yields a contradiction.

(ii)$\Rightarrow$(i)\ If the group $G$ has a nontrivial abelian factor
group, then the subalgebra $\kk[x_1,\dots,x_n]^{G}$ is not
saturated in $\kk[x_1,\dots,x_n]$ (Theorem \ref{T1}). Hence the
subfield $\kk(x_1,\dots,x_n)^{G}$ is also not saturated in
$\kk(x_1,\dots,x_n).$ In the case when $G$ admits a surjective
homomorphism onto the group $I_{120}$ we need an auxiliary lemma.
It is a well-known statement from representation theory, but having no reference we 
give a short proof.

\begin{lemma}\label{known}
Let $\kk$ be a field of characteristic zero and $G\subset
\rm{GL}_{n}(\kk)$ be a finite subgroup. Then every irreducible
representation of the group $G$ can be realized as a
subrepresentation in the $G$-module $\kk[x_1,\dots,x_n]$.
\end{lemma}
\begin{proof}
Since the subspace of $g$-fixed vectors in $\kk^{n}$ is proper for
every $g\in G, \ g\not= 1$, there exists a vector $v\in \kk^{n}$
for which the mapping $\gamma _{v}: G\to \kk^{n}, \ \gamma
_{v}(g)=gv$ is injective. Let $F(G)$ be the space of $\kk$-valued
functions on the group $G$  with the canonical structure of
$G$-module. The natural homomorphism
$$\gamma ^{\star}_{v}: \kk[x_1,\dots,x_n] \to F(G), \ \ (\gamma
^{\star}_{v})(g)=f(\gamma _{v}(g))$$ is $G$-equivariant and
surjective. On the other hand, the $G$-module $F(G)$ is 
dual to the group algebra of the group $G$ and therefore
every simple $G$-module is contained in $F(G)$ with nonzero
multiplicity.
\end{proof}

 Consider a $G$-submodule $U=\langle p,q\rangle$ in $\kk[x_1,\dots,x_n]$
which is isomorphic to the simple two-dimensional
$I_{120}$-module.

 \begin{lemma}\label{L1}
 Let $\kk$ be a field and $H\subseteq \rm{GL}_{2}(\kk)$
 be a finite subgroup containing a non-scalar matrix. Then the
 subfield $\kk(y_{1}, y_{2})^{H}$ is not saturated in $\kk(y_{1},
 y_{2}).$
\end{lemma}

\begin{proof}
Consider the finite set of elements $\{ \frac{g\cdot y_{1}}{g\cdot
y_{2}}: g\in G\} $. At least one  of elementary symmetric
polynomials $\sigma _{i}$ of these elements is non-constant, 
and a generative element for $\sigma _{i}\in \kk(y_{1},
y_{2})^{H}$ is $\frac{y_{1}}{y_{2}} \not\in \kk(y_{1},
y_{2})^{H}.$

\end{proof}

Applying Lemma~\ref{L1} to the image $H$ of $G$ in
$\rm{GL}(U)$, we get the statement.
\end{proof}

\begin{corollary} \label{SC}
Let $G\subset \GL_{n}(\kk)$ be a subgroup isomorphic to
$A_5$. Then the subfield $\kk(x_1,\dots,x_n)^G$ is
saturated in $\kk(x_1,\dots,x_n)$.
\end{corollary}

 It follows from Corollary~\ref{SC} and Example~\ref{ex1} that in the case of
subfields of invariants the saturation property depends not only on the field
$\kk$ and the group $G$, but also on the representation of $G$ by means
of automorphisms of the field $\kk(x_1,\dots,x_n)$.

\smallskip

Since the group $A_5$ admits a faithful three-dimensional representation, Corollary~\ref{SC}
allows for $n\ge 3$ to construct a saturated subfield in $\kk(x_1,\dots,x_n)$ with $\kk(x_1,\dots,x_n)$
being algebraic over this subfield. Clearly, a subfield $\kk\subset L\subseteq\kk(x_1)$ is saturated if and only if
$L=\kk(x_1)$. It remains to consider the case $n=2$. Lemma~\ref{L1} and Theorem~\ref{T1}
imply that for a proper finite subgroup $G\subset\GL_2(\kk)$
the subfield of invariants $\kk(x_1,x_2)^G$ in not saturated. Nevertheless, there is a saturated subfield 
of invariants in $\kk(x_1,x_2)$. As shown in~\cite{MBD}, see also \cite[p.~78]{SY}, 
there is a subgroup $F$ of order 1080 in the group $\SL_3(\kk)$  such that the factor group $F/Z$, where $Z$ is the center of $\SL_3(\kk)$, is
isomorphic to the alternating group $A_6$. This defines an effective action of the group $A_6$ 
on the projective plane $\PP^2$ and on the field $\kk(x_1,x_2)$. By Theorem~\ref{T2}, the subfield 
$\kk(x_1,x_2)^{A_6}$ is saturated in $\kk(x_1,x_2)$. 

\smallskip

 Finally, we give three examples which show that saturation
of the subalgebra of invariants does not imply saturation of the
field of invariants.

\begin{example}
 Consider the following subgroups:
\begin{enumerate}
\item[$\bullet$] $G=I_{120}\subset \rm{GL}_{2}(\CC)$;
\item[$\bullet$] $G=\langle \theta \rangle \subset
\rm{GL}_{2}(\RR)$, where $\theta$ is the $120^{\circ}$-rotation of $\RR^{2}$; 
\item[$\bullet$] $G=\langle \tau \rangle \subset
\rm{GL}_{2}(\kk)$, where $\tau (x_{1})=x_{2}, \ \tau
(x_{2})=x_{1}$, and $\kk$ is a field of characteristic two.
\end{enumerate}
\end{example}
 
 In each of these cases the subalgebra $\kk[x_{1}, x_{2}]^{G}$ is
saturated in $\kk[x_{1}, x_{2}]$ (Theorem~\ref{T1}), but the
subfield $\kk(x_{1}, x_{2})^{G}$ is not saturated
in $\kk(x_{1}, x_{2})$ (Lemma~\ref{L1}).

\smallskip

The authors are grateful to Yu.G.~Prokhorov for the reference to~\cite{SY}.

\end{document}